\documentclass[12pt,oneside,english]{amsart}
\usepackage[latin1]{inputenc}
\usepackage{amssymb}

\theoremstyle{plain}
\newtheorem{theorem}{Theorem}

\newtheorem{lemma}[theorem]{Lemma}

\theoremstyle{definition}

\newtheorem{example}[theorem]{Example}

\theoremstyle{remark}

\usepackage{babel}

\makeatother
\begin{document}
\baselineskip=17pt

\title[Tangent power sums]
{Tangent power sums and their applications}

\author{Vladimir Shevelev}
\address{Department of Mathematics \\Ben-Gurion University of the
 Negev\\Beer-Sheva 84105, Israel. e-mail: shevelev@bgu.ac.il}
 \author{Peter J. C. Moses}
 \address{United Kingdom. e-mail: mows@mopar.freeserve.co.uk}


\begin{abstract}
For integer $m, p,$ we study tangent power sum
$\sum^m_{k=1}\tan^{2p}\frac{\pi k}{2m+1}.$ We give recurrent, asymptotical and explicit formulas
 for these polynomials and indicate their connections with Newman's digit sums in
  base $2m.$
\end{abstract}

\maketitle
\section{Introduction}
Everywhere below we suppose that $n\geq1$ is an odd number and  $p$ is a positive
 integer. In the present paper we study tangent power sum
 of the form
\begin{equation}\label{1}
 \sigma(n, p)=\sum^{\frac{n-1}{2}}_{k=1}\tan^{2p}\frac{\pi k}{n}.
\end{equation}
In 2002, Chen \cite{1} found formulas for $\sigma(n, p)$ in case $p\leq5$ as polynomials in $n.$
 In 2007-2008, Shevelev \cite{12} and Hassan \cite{4} independently proved the following
  statements:
 \begin{theorem}\label{t1}
 For every  $p,$  $\sigma(n,p)$ is integer and multiple of $n.$
 \end{theorem}
 \begin{theorem}\label{t2}
 For a fixed $p,$ $\sigma(n,p)$ is a polynomial in $n$ of degree $2p$ with the leading term
 \begin{equation}\label{2}
 \frac {2^{2p-1}(2^{2p}-1)}{(2p)!}|B_{2p}|n^{2p},
 \end{equation}
where $B_{2p}$ is Bernoulli number.
 \end{theorem}
 Hassan \cite{4} proved these results (see his Theorem 4.3 and formula 4.19), using
 a sampling theorem associated with the second-order discrete eigenvalue problem.

 Shevelev \cite{12} (see his Remark 2 and Remark 1) used some elementary arguments
 including the best-known Littlewood expression for the power sum of elementary
  polynomials in a determinant form \cite{5}.

  In this paper we give another proof of these two theorems.
   Besides, we find several other representations and identities involving
   $\sigma(n,p)$ and numerical results for them. We give applications
   of $\sigma(n,p)$ in digit theory (Section 5). In the conclusive Section 7, using
   the digit interpretation and a combinatorial idea, we found an explicit
    expression for $\sigma(n,p)$ (Theorem \ref{t7}).
 \section{Proof of Theorem 1}
 Denote $\omega=e^{\frac{2\pi i}{n}}.$ Note that
  \begin{equation}\label{3}
 \tan \frac{\pi k}{n}=i\frac{1-\omega^k}{1+\omega^k}=-i\frac{1-\omega^{-k}}{1+\omega^{-k}},\;\; \tan^2 \frac{\pi k}{n}=\frac{1-\omega^{-k}}{1+\omega^k}\frac{1-\omega^{k}}{1+\omega^{-k}}
 \end{equation}
 and, for the factors of $\tan^2 \frac{\pi k}{n},$ we have
 \begin{equation}\label{4}
 \frac{1-\omega^{-k}}{1+\omega^k}=\frac{(-\omega^{k})^{n-1}-1}{(-\omega^{k})-1}=\sum_{j=0}^{n-2}(-\omega^k)^j,\;\;
 \frac{1-\omega^{k}}{1+\omega^{-k}}=\sum_{j=0}^{n-2}(-\omega^{-k})^j.
 \end{equation}
 Since $\tan\frac{\pi k}{n}=-\tan\frac{\pi(n-k)}{n},$ then we have
 \begin{equation}\label{5}
 2\sigma(n, p)=\sum^{n-1}_{k=1}\tan^{2p}\frac{\pi k}{n}
 \end{equation}
 and, by (\ref{3})-(\ref{5}),

 $$2\sigma(n, p)=\sum^{n-1}_{k=1}(\sum_{j=0}^{n-2}(-\omega^k)^j)^p(\sum_{j=0}^{n-2}(-\omega^{-k})^j)^p=$$

 $$\sum^{n-1}_{k=1}
 (\prod_{l=0}^{p-1}\sum_{j=0}^{n-2}(-\omega^k)^j\prod_{l=0}^{p-1}\sum_{j=0}^{n-2}(-\omega^{-k})^j)=$$
  \begin{equation}\label{6}
  =\sum^{n-1}_{k=1}(\prod_{t=0}^{2p-1}\sum_{j=0}^{n-2}(-\omega^{(-1)^tk})^j).
 \end{equation}
 Furthermore, we note that
  \begin{equation}\label{7}
 (n-1)^t\equiv (-1)^t \pmod n.
 \end{equation}
 Indeed, it is evident for odd $t.$ If $t$ is even and $t=2^hs$ with odd $s,$ then
 $$(n-1)^t-(-1)^t=((n-1)^s)^{2^h}-((-1)^s)^{2^h}=$$ $$((n-1)^s-(-1)^s)((n-1)^s+(-1)^s)((n-1)^{2s}+$$ $$(-1)^{2s})\cdot...\cdot((n-1)^{2^{h-1}s}
 +(-1)^{2^{h-1}s}),$$
 and, since $(n-1)^s+1\equiv0\pmod n,$ we are done.
 Using (\ref{7}), we can write (\ref{6}) in the form (we sum from $k=0,$ adding the zero summand)
 \begin{equation}\label{8}
2\sigma(n, p)=\sum^{n-1}_{k=0}\prod^{2p-1}_{t=0}
(1-\omega^{k(n-1)^t}+\omega^{2k(n-1)^t}-...
-\omega^{(n-2)k(n-1)^t}).
 \end{equation}
 Considering $0,1,2,...,n-2$ as digits in the base $n-1,$ after the multiplication of factors of the product in (\ref{8}) we obtain summands of the form $(-1)^{s(r)}\omega^{kr},\;\; r=0,...,(n-1)^{2p}-1,$ where $s(r)$ is the digit sum of $r$ in the base $n-1.$ Thus we have
 \begin{equation}\label{9}
2\sigma(n,p)=\sum^{n-1}_{k=0}\sum^{(n-1)^{2p}-1}_{r=0}(-1)^{s(r)}\omega^{kr}=
\sum^{(n-1)^{2p}-1}_{r=0}(-1)^{s(r)}
\sum^{n-1}_{k=0}(\omega^k)^r.
 \end{equation}
 \newpage
 However,
 $$\sum^{n-1}_{k=0}(\omega^k)^r=\begin{cases}n,\;
if\; r\equiv0 \pmod n\\
0,\; otherwise.\end{cases}$$

Therefore, by (\ref{9}),
\begin{equation}\label{10}
2\sigma(n,p)=n\sum^{(n-1)^{2p}-1}_{r=0,\; n|r}(-1)^{s(r)}
 \end{equation}
 and, consequently, $2\sigma(n,p)$ is integer multiple of $n.$ It is left to show that the right hand side of (\ref{10}) is even. It is sufficient to show that the sum contains even number of summands. The number of summands is $$1+\lfloor\frac{(n-1)^{2p}}{n}\rfloor=1+\frac{(n-1)^{2p}-1}{n}=$$
 $$1+\sum_{l=0}^{2p-1}(-1)^l\binom{2p}{l}n^{2p-1-l}\equiv1+\sum_{l=0}^{2p-1}(-1)^l\binom{2p}{l} \pmod2=$$ $$1-(-1)^{2p}\binom{2p}{2p}=0.$$
 This completes proof of the theorem. \;\;\;\;\;\;\;\;\;$\square$

 \section{Proof of Theorem 2}
 We start with a construction close to one in \cite{15} (see also \cite{12}, Remark 2). As is well known,
$$\sin n\alpha=\sum_{i=0}^{\frac{n-1}{2}}(-1)^i\binom{n}{2i+1}\cos^{n-(2i+1)}\alpha\sin^{2i+1}\alpha,$$
or
$$\sin n\alpha=\tan \alpha \cos^n \alpha\sum_{i=0}^{\frac{n-1}{2}}(-1)^i\binom{n}{2i+1}\tan^{2i}\alpha.$$
Put here $\alpha=\frac{k\pi}{n},\; k=1,2,...,\frac{n-1}{2}.$ Since $\tan\alpha\neq0, \;\cos\alpha\neq0,$ then
$$0=\sum_{i=0}^{\frac{n-1}{2}}(-1)^i\binom{n}{2i+1}\tan^{2i}\alpha=$$
$$(-1)^{\frac{n-1}{2}}(\tan^{n-1}\alpha-\binom{n}{n-2}\tan^{n-3}\alpha+...-$$ $$(-1)^{\frac{n-1}{2}}\binom{n}{3}\tan^2\alpha+(-1)^{\frac{n-1}{2}}\binom{n}{1}).$$
This means that the equation
\begin{equation}\label{11}
\lambda^{\frac{n-1}{2}}-\binom{n}{2}\lambda^{\frac{n-3}{2}}+\binom{n}{4}\lambda^{\frac{n-5}{2}}-...+(-1)^{\frac{n-1}{2}}\binom{n}{n-1}=0
 \end{equation}
 has $\frac{n-1}{2}$ roots: $\lambda_k=\tan^2\frac{k\pi}{n},\;k=1,2,...,\frac{n-1}{2}.$
 Note that (\ref{11}) is the characteristic equation for the following difference equation
 \newpage
$$y(p)=\binom{n}{2}y(p-1)-\binom{n}{4}y(p-2)+...-$$

  \begin{equation}\label{12}
 (-1)^{\frac{n-1}{2}}\binom{n}{n-1}y(p-\frac{n-1}{2})
 \end{equation}

 which, consequently, has a private solution
 $$y(p)=\sum_{k=1}^{\frac{n-1}{2}}(\tan^2\frac{k\pi}{n})^p=\sigma(n,p).$$
 Now, using Newton's formulas for equation (\ref{11}),
 $$\sigma(n,1)=\binom{n}{2},$$
 $$\sigma(n,2)=\binom{n}{2}\sigma(n,1)-2\binom{n}{4},$$
 \begin{equation}\label{13}
 \sigma(n,3)=\binom{n}{2}\sigma(n,2)-\binom{n}{4}\sigma(n,1)+3\binom{n}{6}, \;etc.
 \end{equation}
 we conclude that $\sigma(n,p)$ is a polynomial in $n$ of degree $2p.$ Note that, by induction, all these polynomials are integer-valued and thus we have another independent proof of Theorem \ref{t1}. To find the leading terms of these polynomials, we make some transformations of (\ref{1}). Put $\frac{n-1}{2}=m.$ Changing in (\ref{1}) the order of summands $(l=m-k)$ and noting that $$\frac{(m-l)\pi}{2m+1}+\frac{(2l+1)\pi}{4m+2}=\frac{\pi}{2},$$ we have
 \begin{equation}\label{14}
\sigma(n,p)=\sum_{l=0}^{m-1}\cot^{2p}\frac{(2l+1)\pi}{4m+2}.
\end{equation}
Further we have
$$\sigma(n,p)=\sum_{0\leq l\leq \sqrt {m}}\cot^{2p}\frac{(2l+1)\pi}{4m+2}+$$
\begin{equation}\label{15}
\sum_{\sqrt{m}<l\leq m-1} \cot^{2p}\frac{(2l+1)\pi}{4m+2}=\Sigma_1+\Sigma_2.
\end{equation}
Let $p>1.$ Let us estimate the second sum $\Sigma_2.$ The convexity of $\sin x$ on $[0,\frac{\pi}{2}]$ gives the inequality $\sin x\geq \frac{2}{\pi}x.$ Therefore, for summands in the second sum, we have
$$\cot^{2p}\frac{(2l+1)\pi}{4m+2}<\sin^{-2p}\frac{(2l+1)\pi}{4m+2}<$$ $$(\frac{2m+1}{2l+1})^{2p}
<(\frac{2m+1}{2\sqrt{m}+1})^{2p}<m^p.$$
\newpage
This means that $\Sigma_2<m^{p+1}<m^{2p}$ and not influences on the leading term. Now note that, evidently,
$$ \frac {(2l+1)\pi}{4m+2}cot \frac {(2l+1)\pi}{4m+2}\rightarrow1$$
uniformly over $l\leq\sqrt{m}.$ Thus

$$\Sigma_1=\sum_{0\leq l\leq \sqrt {m}}(\frac{(4m+2)}{(2l+1)\pi})^{2p}+\alpha(m)=$$ $$(\frac{(4m+2)}{\pi})^{2p}\sum_{0\leq l\leq \sqrt {m}}\frac{1}{(2l+1)^{2p}}+\alpha(m), $$
where $\alpha(m)\leq \varepsilon\sqrt{m}.$ Thus the coefficient of the leading term of the polynomial $\sigma(n,p)$ is
$$\lim_{m\rightarrow\infty}\frac{\Sigma_1}{n^{2p}}=(\frac{2}{\pi})^{2p}\sum_{l=0}^{\infty}\frac{1}{(2l+1)^{2p}}=$$$$
(\frac{2}{\pi})^{2p}(\zeta(2p)-\sum_{l=1}^{\infty}\frac{1}{(2l)^{2p}})= $$$$
(\frac{2}{\pi})^{2p}(\zeta(2p)-\frac{1}{2^{2p}}\zeta(2p))=\frac{2^{p}(2^{2p}-1)}{\pi^{2p}}\zeta(2p). $$
It is left to note that, by very known formula, $\zeta(2p)=\frac{|B_{2p}|2^{2p-1}\pi^{2p}}{(2p)!},$ we find that the leading coefficient is defined by formula (\ref{2}). \;\;\;\;\;\;\;\;\;\;\;\;\;\;\;\;\;\;\;\;\;$\square$

\section{Several numerical results}
Since, by (\ref{1}), $\sigma(1,p)=0,$ then $\sigma(n,p)\equiv0\pmod {n(n-1)}.$ Put
$$\sigma^*(n,p)=2\sigma(n,p)/(n(n-1)).  $$

  By formulas (\ref{13}), the first polynomials $\{\sigma^*(n,p)\}$ are
$$\sigma^*(n,1)=1,$$
$$\sigma^*(n,2)=\frac{n^2+n}{3}-1,$$
$$\sigma^*(n,3)=\frac{2(n^2+n)(n^2-4)}{15}+1,$$
$$\sigma^*(n,4)=\frac{(n^2+n)(17n^4-95n^2+213)}{315}-1,$$
$$\sigma^*(n,5)=\frac{2(n^2+n)(n^2-4)(31n^4-100n^2+279)}{2835}+1,\;etc.$$
As well known (cf. Problem 85 in \cite{8}), the integer-valued polynomials have integer coefficients in the binomial basis $\{\binom {n}{k}\}.$ The first integer-valued polynomials $\{\sigma(n,p)\}$ represented in binomial basis have the form
$$\sigma(n,1)=\binom{n}{2},$$
\newpage
$$\sigma(n,2)=\binom{n}{2}+6\binom{n}{3}+4\binom{n}{4},$$
$$\sigma(n,3)=\binom{n}{2}+24\binom{n}{3}+96\binom{n}{4}+120\binom{n}{5}+48\binom{n}{6},$$
$$\sigma(n,4)=\binom{n}{2}+78\binom{n}{3}+836\binom{n}{4}+3080\binom{n}{5}+5040\binom{n}{6}+3808\binom{n}{7}
+1088\binom{n}{8},$$
etc.\newline
\indent Note that the recursion (\ref{12}) presupposes a fixed $n.$ In general, by (\ref{12}), we have
$$\sigma(n,p)=\binom{n}{2}\sigma(n,p-1)-\binom{n}{4}\sigma(n,p-2)+...-$$
\begin{equation}\label{16}
(-1)^{\frac{n-1}{2}}\binom{n}{n-1}\sigma(n,p-\frac{n-1}{2}),\;p\geq\frac{n-1}{2}.
 \end{equation}
  Since from (\ref{1}) $\sigma(n,0)=\frac{n-1}{2},\;n=3,5,...,$ then, calculating other initials by (\ref{13}), we have the recursions:
 $$\sigma(3,p)=3\sigma(3,p-1), \; p\geq1,\; \sigma(3,0)=1;$$
 $$\sigma(5,p)=10\sigma(5,p-1)-5\sigma(5,p-2),\; p\geq2,\; \sigma(5,0)=2,\; \sigma(5,1)=10;  $$
 $$\sigma(7,p)=21\sigma(7,p-1)-35\sigma(7,p-2)+7\sigma(7,p-3),\; p\geq3,\;$$
 $$\sigma(7,0)=3,\; \sigma(7,1)=21,\;\sigma(7,2)=371; $$
 $$\sigma(9,p)=36\sigma(9,p-1)-126\sigma(9,p-2)+84\sigma(9,p-3)-9\sigma(9,p-4),\;p\geq4,\;$$
 $$\sigma(9,0)=4,\; \sigma(9,1)=36,\;\sigma(9,2)=1044,\; \sigma(9,3)=33300; \;etc.$$

 Thus
 \begin{equation}\label{17}
 \sigma(3,p)=3^p,
 \end{equation}
 and a few terms of the other sequences $\{\sigma(n,p)\}$ are
 $$n=5)\;\;\;\; 2,10,90,850,8050,76250,722250,6841250,64801250,$$ $$613806250,5814056250,...;$$
 $$ n=7)\;\;\; 3,21,371,7077,135779,2606261,50028755,960335173,$$$$18434276035,353858266965,6792546291251,...;$$
  $$n=9)\;\;\;\;\;\;\;\;\;\;4,36,1044,33300,1070244,34420356,1107069876,$$ $$35607151476,1145248326468,36835122753252,...;$$
$$n=11)\;\;\;\;\;\;\; 5,55,2365,113311,5476405,264893255,12813875437,$$
 $$619859803695,29985188632421,1450508002869079,...\;.$$
\newpage
\section{Applications to digit theory}
For $x\in \mathbb{N}$ and odd $n\geq 3$, denote by $S_{n}(x)$ the sum

\begin{equation}\label {18}
S_{n}(x)=\sum_{0\leq r < x: \;\;r\equiv 0\pmod n}(-1)^{s_{n-1}(r)},
\end{equation}
where $s_{n-1}(r)$ is the digit sum of $r$ in base $n-1.$\newline
Note that, in particular, $S_{3}(x)$ equals the difference between the numbers of multiples of 3 with even and odd binary digit sums (or multiples of 3 from sequences A001969 and A000069 in \cite{14}) in interval $[0,x).$

Leo Moser (cf. \cite{7}, Introduction) conjectured that always

\begin{equation}\label{19}
S_{3}(x)>0.
\end{equation}

 Newman \cite{7} proved this conjecture. Moreover, he obtained
the inequalities

\begin{equation}\label{20}
\frac{1}{20} < S_{3}(x)x^{-\lambda}< 5,
\end{equation}
where
\begin{equation}\label{21}
\lambda=\frac{\ln 3}{\ln 4}=0.792481...\;.
\end{equation}
In connection with these remarkable Newman results, the qualitative result (\ref{19}) we call a weak Newman phenomenon (or Moser-Newman phenomenon), while an estimating result of
 the form (\ref{20}) we call a strong Newman phenomenon.\newline
\indent In 1983, Coquet \cite{2} studied a very complicated continuous and nowhere differentiable fractal
function $F(x)$ with period 1 for
which

\begin{equation}\label{22}
S_{3}(3x)=x^\lambda F\left(\frac{\ln x}{\ln
4}\right)+\frac{\eta(x)}{3},
\end{equation}

where

\begin{equation}\label{23}
\eta(x)=\begin{cases} 0,\;\; if \; x \;\; is\;\; even,\\
(-1)^{s_2(3x-1)}, \;\; if \;\; x \;\; is\;\; odd.\end{cases}
\end{equation}

He obtained that

\begin{equation}\label{24}
\limsup_{x\rightarrow
\infty, \;x\in\mathbb{N}}S_{3}(3x)x^{-\lambda}=\frac{55}{3}\left(\frac{3}{65}\right)^\lambda=1.601958421\ldots\;,
\end{equation}
\newpage
\begin{equation}\label{25}
\liminf_{x\rightarrow\infty, \;x\in\mathbb{N}}S_{3}(3x)x^{-\lambda}=\frac{2\sqrt{3}}{3}
=1.154700538\ldots\;.
\end{equation}

In 2007, Shevelev \cite{11} gave an elementary proof of Coquet's formulas (\ref{24})-(\ref{25}) and his sharp estimates in the form
\begin{equation}\label{26}
\frac{2\sqrt{3}}{3}x^\lambda\leq
S_{3}(3x)\leq\frac{55}{3}\left(\frac {3}{65}\right)^\lambda
x^\lambda,\;\;x\in\mathbb{N}.
\end{equation}
Besides, Shevelev showed that the sequence $\{(-1)^{s_2(n)}(S_3(n)-3S_3
(\lfloor n/4\rfloor))\}$, is periodic with period 24 taking the values
$-2,-1,0,1,2.$ This gives a simple recursion for $S_3(n).$
In 2008, Drmota and Stoll \cite{3} proved a generalized weak Newman phenomenon, showing that (\ref{19}) is valid for sum (\ref{18}) for every $n\geq3,$ at least beginning with $x\geq x_0(n).$ Our proof of Theorem \ref{t1} allows to consider a strong form of this generalization, but yet only in ``full" intervals in even base $n-1$ of the form $[0, (n-1)^{2p})$ (see also preprint
of Shevelev \cite{12}).
\begin{theorem}\label{t3}
For $x_{n,p}=(n-1)^{2p},\; p\geq1,$ we have

\begin{equation}\label{27}
S_n(x_{n,p})\sim \frac{2}{n}x_{n,p}^\lambda \;\;\sigma(n,p)\sim x_{n,p}^\lambda\;\; (p\rightarrow\infty),
\end{equation}
where
\begin{equation}\label{28}
\lambda=\lambda_n=\frac{\ln\cot(\frac{\pi}{2n})}{\ln(n-1)}.
\end{equation}
\end{theorem}
\begin{proof} According to (\ref{10}) and (\ref{18}), we have
\begin{equation}\label{29}
S_n(x_{n,p})=\frac{2}{n}\sigma(n,p),\; p\geq1.
\end{equation}
Thus, choosing the maximal exponent in (\ref{1}) as
$p\rightarrow\infty,$ we find

$$S_n(x_{n,p})\sim\frac{2}{n}\tan^{2p} \frac {(n-1)\pi}{2n}=$$ $$\frac{2}{n}\cot^{2p}\frac{\pi}{2n}=\exp(\ln\frac{2}{n}+2p\ln\cot\frac{\pi}{2n})=$$
\begin{equation}\label{30}
\exp(\ln\frac{2}{n}+2p\lambda\ln(n-1))=\exp(\ln\frac{2}{n}+\ln x_{n,p}^\lambda)=\frac{2}{n}x_{n,p}^\lambda.
\end{equation}
\end{proof}
In particular, in the cases of $n=3,5,7,9,11$ we have $\lambda_3=\frac{\ln 3}{\ln 4}=0.79248125...,$
$\lambda_5= 0.81092244...,\; \lambda_7=0.82452046...,\;\lambda_9= 0.83455828..., \lambda_{11}=
0.84230667...$ respectively.

Show that
\begin{equation}\label{31}
1-\frac{\ln\frac \pi 2}{\ln(n-1)}\leq\lambda_n\leq1-\frac{\ln\frac \pi 2}{\ln(n-1)}+\frac{1}{(n-1)\ln(n-1)}.
\end{equation}
\newpage
Indeed, by the convexity of $\cos x$ on $[0,\frac{\pi}{2}],$  $\cos x\geq 1-\frac{2}{\pi}x,$ and, therefore, $\cos\frac{\pi}{2n}\geq1-\frac{1}{n}.$ Using also that $\tan\frac{\pi}{2n}\geq\frac{\pi}{2n}\geq\sin\frac{\pi}{2n},$ we have

$$
\frac{2}{\pi}(n-1)
\leq\cot\frac{\pi}{2n}\leq \frac{2}{\pi}n
$$
and, by (\ref{28}),
$$1-\frac{\ln\frac \pi 2}{\ln(n-1)}\leq\lambda_n\leq1-\frac{\ln\frac \pi 2}{\ln(n-1)}+\frac{\ln(1+\frac{1}{n-1})}{\ln(n-1)} $$
which yields (\ref{31}), since, for $n\geq3,$ $\ln(1+\frac{1}{n-1})<\frac{1}{n-1}.$ Finally, let us show the monotonic increasing of $\lambda_n.$  For function $f(x)=\frac{\ln\cot(\frac{\pi}{2x})}{\ln(x-1)},$ we have
\begin{equation}\label{32}
 \ln(x-1)f'(x)=\frac{\pi}{x^2\sin \frac{\pi}{x}}-\frac{f(x)}{x-1}.
\end{equation}
As in (\ref{31}), we also have
\begin{equation}\label{33}
 f(x)\leq1-\frac{\ln\frac \pi 2}{\ln(x-1)}+\frac{1}{(x-1)\ln(x-1)}.
\end{equation}
On the other hand, since $\sin\frac{\pi}{x}\leq\frac{\pi}{x},$ then
$$\frac{\pi(x-1)}{x^2\sin \frac{\pi}{x}}\geq1-\frac{1}{x}, $$
and, by (\ref{32}), in order to show that $f'(x)>0,$ it is sufficient to prove that $f(x)<1-\frac{1}{x},$ or, by (\ref{33}), to show that
$$1-\frac{\ln\frac \pi 2}{\ln(x-1)}+\frac{1}{(x-1)\ln(x-1)}<1-\frac{1}{x},$$
or
$$\frac{\ln(x-1)}{x}+\frac{1}{x-1}<\ln\frac{\pi}{2}. $$
This inequality holds for $x\geq7,$ and since $\lambda_3<\lambda_5<\lambda_7,$ then the monotonicity of $\lambda_n$ follows. Thus we have the monotonic strengthening of the strong form of Newman-like phenomenon for the base $n-1$ in
 the considered intervals.

\section{An identity}
Since (\ref{29}) was proved for
   $x_{n,p}=(n-1)^{2p},\; p\geq1,$ then, by (\ref{16}), for $S_n(x_{n,p})$
   in the case $p\geq \frac{n+1}{2},$ we have the relations
$$\sum_{k=0}^{\frac{n-1}{2}}(-1)^k\binom{n}{2k}\sigma(n, p-k)= $$
$$ \sum_{k=0}^{\frac{n-1}{2}}(-1)^k\binom{n}{2k}S_n((n-1)^{2p-2k})=0.$$
\newpage
In case $p=\frac{n-1}{2}$ the latter relation does not hold. Let us show that in this
 case we have the identity
$$ \sum_{k=0}^{\frac{n-1}{2}}(-1)^k\binom{n}{2k}S_n((n-1)^{n-2k-1})=(-1)^n,$$
or, putting $n-2k-1=2j,$ the identity
\begin{equation}\label{34}
\sum_{j=0}^{\frac{n-1}{2}}(-1)^j\binom{n}{2j+1}S_{n}((n-1)^{2j})=1.
\end{equation}
Indeed, in case $j=0,$ we, evidently, have $S_{n}(1)=1,$ while, formally,
by (\ref{29}), for $p=0,$ we obtain $"S_{n}(1)=\frac{2}{n}\sigma(n,\;0)=
\frac{2}{n}\frac{n-1}{2}=\frac{n-1}{n}",$ i.e., the error is $-\frac{1}{n},$
and the error in the corresponding sum is $n(-\frac{1}{n})=-1.$ Therefore, in the
latter formula, instead of 0, we have 1. Note that (\ref{34}) one can rewrite
 also in the form
$$\sum_{j=1}^{\frac{n-1}{2}}(-1)^{j-1}\binom{n}{2j+1}\sigma(n, j)=
\binom{n}{2}. $$
\section{Explicit combinatorial representation}
 In its turn, the representation (\ref{29}) allows to get an explicit combinatorial representation for $\sigma(n,p). $ We need three lemmas.
 \begin{lemma}\label{L4} $($\cite{10}, p. $215$ $)$ The number of compositions $C(m,n,s)$
 of $m$ with $n$ positive parts not exceeding $s$ is given by formula
 \begin{equation}\label{35}
C(m,n,s)=\sum_{j=0}^{\min(n,\lfloor\frac{m-n}{s}\rfloor)}(-1)^j\binom{n}{j}\binom{m-sj-1}{n-1}.
\end{equation}
 \end{lemma}
 Since, evidently, $C(m,n,1)=\delta_{m,n},$ then, as a corollary, we have the identity
 \begin{equation}\label{36}
\sum_{j=0}^{\min(n,m-n)}(-1)^j\binom{n}{j}\binom{m-j-1}{n-1}=\delta_{m,n}.
\end{equation}
\begin{lemma}\label{L5} The number of compositions $C_0(m,n,s)$
 of $m$ with $n$ nonnegative parts not exceeding $s$ is given by formula \begin{equation}\label{37}
C_0(m,n,s)=\begin{cases}C(m+n,n,s+1),\;
if\;m\geq n\geq1,\;s\geq2,\\ \sum_{\nu=1}^{m} C(m,\nu,s)\binom{n}{n-\nu},
\;if\;1\leq m<n,\; s\geq2,\\1,\;if\; m=0,\; n\geq1,\; s\geq0,\\0,\;if\; m>n\geq1,\;s=1,
\\\binom {n}{m}, \;if\;1\leq m\leq n,\;s=1.
\end{cases}
\end{equation}
\newpage
 \end{lemma}
 \begin{proof} \; Let firstly $s\geq2,\; m\geq n\geq1.$ If to diminish on 1 every
 part of a composition of $m+n$ with $n$ positive parts not exceeding $s+1,$ then we
  obtain a composition of $m$ with $n$ nonnegative parts not exceeding $s,$ such
  that zero parts allowed.
  Let, further, $s\geq2,\; 1\leq m<n.$ Consider $C(m,\nu,s)$
   compositions of $m$ with $\nu\leq m$ parts. To obtain $n$ parts, consider $n-\nu$
   zero parts, which we choose in $\binom{n}{n-\nu}$ ways. The summing over $1\leq\nu\leq m$
   gives the required result. Other cases are evident.
   \end{proof}

   Let now $(n-1)^{h}\leq N<(n-1)^{h+1}, \;n\geq3.$ Consider the representation
   of $N$ in the base $n-1:$
$$ N=g_h(n-1)^h+...+g_1(n-1)+g_0, $$
where $g_i=g_i(N),\; i=0,...,h,$ are digits of $N, \;\;0\leq g_i\leq n-2.$
Let
$$s^{e}(N)=\sum_{i\;is\;even} g_i,\;s^{o}(N)=\sum_{i\;is\;odd} g_i.$$
\begin{lemma}\label{L6}
 $N$ is multiple of $n$ if and only if
$s^{o}(N)\equiv s^{e}(N)\pmod n.$
\end{lemma}
\begin{proof} The lemma follows from the evident relation $(n-1)^i\equiv(-1)^i
 \pmod{n},$\newline $i\geq0.$
\end{proof}
   Now we obtain a combinatorial explicit formula for $\sigma(n,p).$
\begin{theorem}\label{t7} For $n\geq3,\;p\geq1,$ we have
$$\sigma(n,p)=\frac{n}{2}\sum_{j=0}^{(n-2)p}((C_0(j,p,n-2))^2+$$
\begin{equation}\label{38}
2\sum_{k=1}^{\lfloor\frac{(n-2)p-j}{n}\rfloor}(-1)^kC_0(j,p,n-2)C_0(j+nk,p,n-2)),
\end{equation}
where $C_0(m,n,s)$ is defined by formula (\ref{37}).
\end{theorem}
\begin{proof} Consider all nonnegative integers $N's$ not exceeding
$(n-1)^{2p}-1,$ which have $2p$ digits $g_i(N)$ in base $n-1$
(the first 0's allowed). Let the sum of digits of $N$
on even $p$ positions be $j,$ while on odd $p$ positions such sum be $j+kn$ with
a positive integer $k.$ Then, by Lemma \ref{L6}, such $N's$ are multiple of $n.$
Since in the base $n-1$ the digits not exceed $n-2,$ then the number of ways
to choose such $N's,$ for $k=0,$ is $(C_0(j,p,n-2))^2.$ In the case $k\geq1,$
 we should also consider the symmetric case when on odd $p$ positions the sum of
  digits of $N$ be $j,$ while on even $p$ positions such sum be $j+kn$ with a positive
integer $k.$ This, for $k\geq1,$ gives $2C_0(j,p,n-2)C_0(j+kn,p,n-2)$ required numbers $N's.$
 Furthermore, since $n$ is odd, then, if $k$ is odd, then
 \newpage $s_{n-1}(N)$ is odd,
  while, if $k$ is even, then $s_{n-1}(N)$ is even. Thus the difference
  $S_n((n-1)^{2p})$ between $n$-multiple $N's$ with even and odd digit sums
  equals
$$S_n((n-1)^{2p})=\sum_{j}((C_0(j,p,n-2))^2+$$
 $$2\sum_{k}(-1)^{k}C_0(j,p,n-2)C_0(j+nk,p,n-2)).$$
 Now to obtain (\ref{38}), note that $0\leq j\leq(n-2)p,$ and, for $k\geq1,$ also
 $j+nk\leq (n-2)p,$ such that $1\leq k\leq\frac{(n-2)p-j}{n},$ and that, by (\ref{29}),
 $\sigma(n,p)=\frac{n}{2}S_n((n-1)^{2p}).$
\end{proof}
\begin{example}\label{e8}
Let $n=5, p=2.$ By Theorem \ref{t7}, we have
$$\sigma(5,2)=2.5\sum_{j=0}^{6}((C_0(j,2,3))^2+$$
\begin{equation}\label{39}
2\sum_{k=1}^{\lfloor\frac{6-j}{3}\rfloor}(-1)^kC_0(j,2,3)C_0(j+5k,2,3)).
\end{equation}
We have
$$C_0(0,2,3)=1,C_0(1,2,3)=2,C_0(2,2,3)=3,$$
 $$C_0(3,2,3)=4,C_0(4,2,3)=3,C_0(5,2,3)=2,C_0(6,2,3)=1.$$
 Thus
 $$\sum_{j=0}^{6}((C_0(j,2,3))^2=44.$$
 In the cases $j=0, k=1$ and $j=1, k=1$ we have
 $$C_0(0,2,3)C_0(5,2,3)=2,\; C_0(1,2,3)C_0(6,2,3)=2.$$
 Thus
$$ 2\sum_{j=0}^6\sum_{k=1}^{\lfloor\frac{6-j}{3}\rfloor}(-1)^k
C_0(j,2,3)C_0(j+5k,2,3))=-8$$
and, by (\ref{39}), we have
$$\sigma(5,2)=2.5(44-8)=90.$$
On the other hand, by (\ref{1}), we directly have
$$\sigma(5, 2)=\sum^{2}_{k=1}\tan^{4}\frac{\pi k}{5}=0.278640...+89.721359...=89.999999...$$
\end{example}
\begin{example}\label{e9}
In case $n=3,$ by Theorem \ref{t7} and formulas (\ref{17}), (\ref{37}), we have
$$3^p=\frac{3}{2}\sum_{j=0}^{p}((C_0(j,p,1))^2+$$
\newpage
$$2\sum_{k=1}^{\lfloor\frac{p-j}{3}\rfloor}(-1)^kC_0(j,p,1)C_0(j+3k,p,1))=$$
$$\frac{3}{2}\sum_{j=0}^{p}(\binom{p}{j}^2+
2\sum_{k=1}^{\lfloor\frac{p-j}{3}\rfloor}(-1)^k\binom{p}{j}\binom{p}{3k+j}.$$
Thus, using well known formula $\sum_{j=0}^{p}(\binom{p}{j}^2=\binom{2p}{p},$
we find the identity
$$\sum_{j=0}^{p}\sum_{k=1}^{\lfloor\frac{p-j}{3}\rfloor}(-1)^k\binom{p}{j}
\binom{p}{3k+j}=3^{p-1}-\frac{1}{2}\binom{2p}{p},$$
or, changing the order of summing,
$$\sum_{k=1}^{\lfloor\frac{p}3\rfloor}(-1)^k\sum_{j=0}^{p-3k}\binom{p}{j}
\binom{p}{3k+j}=3^{p-1}-\frac{1}{2}\binom{2p}{p}.$$
Since (cf.\cite{9},p.8)

\begin{equation}\label{40}
\sum_{j=0}^{p-3k}\binom{p}{j}\binom{p}{3k+j}=\binom{2p}{p+3k},
\end{equation}
then we obtain an identity
\begin{equation}\label{41}
\sum_{k=1}^{\lfloor\frac{p}3\rfloor}(-1)^{k-1}\binom{2p}{p+3k}=
\frac{1}{2}\binom{2p}{p}-3^{p-1},\; p\geq1.
\end{equation}
\end{example}
Note that firstly (\ref{41}) was proved in a quite another way by Shevelev
 \cite{13} (2007) and again proved by Merca \cite{6} (2012).

\bfseries Acknowledgment \mdseries The authors are grateful to Jean-Paul Allouche for the indicating the paper \cite{4}.


\begin{thebibliography}{12}
\bibitem{1} H. Chen, On some trigonometric power sums, \emph{IJMMS},\;
 \bfseries 30 \mdseries, no. 3 (2002), 185-191.
\bibitem{2} J. Coquet, A summation formula related to the binary digits,
\emph{Invent. Math.},\; \bfseries 73 \mdseries (1983),107-115.
\bibitem{3} M. Drmota, and T. Stoll, Newman's phenomenon for generalized Thue-Morse sequence,
\emph{Discrete Math.},\;\bfseries 308\mdseries (2008) no.7, 1191-1208.
\bibitem{4} H. A. Hassan, New trigonometric sums by sampling theorem, \emph{J. Math. Anal. Appl.},\; \bfseries 339 \mdseries (2008), 811-827.
\bibitem{5} J. E. Littlewood, \emph{A University Algebra}, 2nd ed., London, Heinemann, 1958.
\bibitem{6} M. Merca, A note of cosine power sums, \emph{J. Integer Seq.}, \bfseries 15
\mdseries\; (2012), Article 12.5.3
\bibitem{7} D. J. Newman, On the number of binary digits in a multiple of three,
\emph{Proc. Amer. Math. Soc.},\;\bfseries 21 \mdseries (1969),719-721.
\bibitem{8} G. Polya and G. Szeg\"{o}, \emph{Problems and theorems in analysis},
\enskip Vol. 2, Springer-Verlag, 1976.
\bibitem {9} J. Riordan, \emph{Combinatorial identities},\; Welley, 1968
\newpage
 \bibitem{10} V. S. Sachkov, \emph{Introduction to combinatorial methods of
  descrete mathematics},\; Moscow, Nauka, 1982 (In Russian).
\bibitem{11} V. Shevelev, Two algorithms for exact evalution of the Newman
digit sum, and a new proof of Coquet's theorem, \emph{arXiv}, 0709.0885 [math.NT].
\bibitem{12} V. Shevelev, On Monotonic Strengthening of Newman-like Phenomenon on $(2m+1)$-multiples in Base $2m,$ \emph{arXiv},0710.3177 [math.NT].
\bibitem{13} V. Shevelev, A conjecture on primes and a step towards justification, \emph{arXiv},0706.0786 [math.NT].
\bibitem {14}  N.\; J.\; A.\; Sloane,\; \emph{The On-Line Encyclopedia of Integer Sequences}  \; http://oeis.org.
\bibitem {15} A. M. Yaglom and I. M. Yaglom, An elementary proof of the Wallis, Leibniz and Euler
formulas for $\pi$, Uspekhi Matem.NaukVIII (1953), 181-187 (In Russian).
\end{thebibliography}
\end{document}